\documentclass[12pt]{amsart} 

\usepackage{amssymb,amsmath,amsthm, url, xy}
\usepackage{amscd,amssymb,url,colonequals,enumitem,xy}
\xyoption{all}

\usepackage[english]{babel}
\usepackage[latin1]{inputenc}
\usepackage{amsmath}
\usepackage{amsthm}
\usepackage{amssymb}
\usepackage{appendix}
\usepackage{verbatim}
\usepackage{enumitem}
\usepackage{multicol}

\usepackage{cleveref}

\usepackage{mathtools}

\newtheorem{thm}{Theorem}[section]

\newtheorem{prop}[thm]{Proposition}
\newtheorem{lem}[thm]{Lemma}
\newtheorem{lemma}[thm]{Lemma}

\theoremstyle{definition}

\newtheorem{example}[thm]{Example}
\theoremstyle{remark}
\newtheorem{rmk}[thm]{Remark}
\numberwithin{equation}{section} 

\begin{document}
\newcommand{\Sym}{\operatorname{S}}
\newcommand{\GL}{\operatorname{GL}}
\newcommand{\C}{{\mathbb C}}
\newcommand{\bbC}{{\mathbb C}}
\newcommand{\bbG}{{\mathbb G}}
\newcommand{\bbQ}{{\mathbb Q}}
\newcommand{\bbZ}{{\mathbb Z}}
\newcommand{\bbA}{{\mathbb A}}
\newcommand{\N}{{\mathbb N}}
\newcommand{\R}{{\mathbb R}}

\newcommand{\SL}{\operatorname{SL}}
\newcommand{\Span}{\operatorname{Span}}
\newcommand{\Stab}{\operatorname{Stab}}
\renewcommand{\P}{{\mathbb P}}
\newcommand{\bbP}{{\mathbb P}}
\renewcommand{\O}{{\mathcal O}}
\newcommand{\F}{{\mathcal F}}
\newcommand{\I}{{\mathcal I}}
\newcommand{\Mat}{\operatorname{M}}
\newcommand{\Char}{\operatorname{char}}
\newcommand{\tr}{{\rm tr}}

\newcommand{\Q}{\mathbb Q}
\newcommand{\Z}{\mathbb Z}
\newcommand{\G}{\mathbb G}
\renewcommand{\P}{\mathbb P}
\newcommand{\Spec}{\operatorname{Spec}}
\renewcommand{\c}{\subseteq}
\newcommand{\A}{\mathbb A}
\renewcommand{\L}{\mathbb L}
\newcommand{\mc}[1]{\mathcal{#1}}
\newcommand{\cl}{\overline}
\newcommand{\set}[1]{\{#1\}}
\renewcommand{\phi}{\varphi}
\newcommand{\ed}{\operatorname{ed}}
\newcommand{\on}[1]{\operatorname{#1}}
\newcommand{\laur}[1]{#1(\!(t)\!)}
\newcommand{\ang}[1]{\left \langle{#1}\right \rangle}

\author{Zinovy Reichstein}
\address{Department of Mathematics\\
 University of British Columbia\\
 Vancouver, BC V6T 1Z2\\Canada}
 \email{reichst@math.ubc.ca}
\thanks{Zinovy Reichstein was partially supported by
 National Sciences and Engineering Research Council of
 Canada Discovery grant 253424-2017.}

\author{Federico Scavia}
\email{scavia@math.ubc.ca}
\thanks{Federico Scavia was partially supported by a graduate fellowship from the University of British Columbia.}

\subjclass[2010]{Primary 14E08, 14M20}
	
\keywords{Noether Problem, rationality, spinor groups}

\title{The Noether Problem for spinor groups of small rank}

	\begin{abstract}
		Building on prior work of Bogomolov, Garibaldi, Guralnick, Igusa, Kordonski{\u\i}, Merkurjev and others, we show that the Noether Problem for $\on{Spin}_n$ has a positive solution for every $n\leq 14$ over an arbitrary field of characteristic
		$\neq 2$.
	\end{abstract}
	
	\maketitle
			
\section{Introduction}	

Let $k$ be a field, $\overline{k}$ be an algebraic closure of $k$, $G$ be a linear algebraic group defined over $k$, and $\rho \colon G \hookrightarrow \GL(V)$ be a $k$-representation. Assume that $\rho$ is generically free; that is, the scheme-theoretic stabilizer $\Stab_G(v)$ is trivial for a point $v \in V(\overline{k})$ in general position. The Noether Problem asks whether the field of rational $G$-invariants $k(V)^G$ is a purely transcendental extension of $k$. Equivalently, it asks whether the Rosenlicht quotient $V/G$ is rational over $k$.
(For the definition of the Rosenlicht quotient, see \Cref{sect.rosenlicht}.) The following variants of the Noether Problem are also of interest: Is $V/G$ stably rational? Is $V/G$ retract rational? Recall that a $d$-dimensional algebraic variety $X$ is called rational if $X$ is birationally equivalent to the affine space $\bbA^d$, stably rational if $X \times \bbA^r$ is birationally equivalent to $\bbA^{d + r}$ for some $r \geqslant 0$ and retract rational if the identity morphism $\operatorname{id} \colon V/G \to V/G$, viewed as a rational map, can be factored through the affine space $\bbA^m$ for some $m \geqslant d$. 

By the no-name lemma~\cite[Lemma 2.1]{reichstein2006birational} the answer to 
the Noether Problem for stable and retract rationality depends only on the group $G$ and not on the choice of 
the representation $V$. Following A.~Merkurjev~\cite{merkurjev2017invariants}, we will say that the classifying stack $BG$ is stably 
(respectively, retract) rational if $V/G$ is stably (respectively, retract) rational for some (and thus every) 
generically free representation $G \hookrightarrow \GL(V)$. We will also say that $BG$ and $BH$ are stably birationally equivalent if $V/G$ and $W/H$ are stably birationally equivalent, where  $H \hookrightarrow \GL(W)$ is a generically free representation of $H$. This terminology is 
related to the fact that $V/G$ can be thought of as an approximation to $BG$. Note that
for us ``$BG$ is stably rational" will be a convenient short-hand for ``the Noether Problem for stable rationality has a positive answer for $G$"; we will not actually work with stacks in this paper. 

In the case where $G$ is a finite group, and $V$ is the regular representation of $G$, the question of rationality of $V/G$
was posed by E.~Noether in the context of her work on the Inverse Galois Problem~\cite{noether1917gleichungen}. 
For a finite group $G$, $V/G$ may not be stably (or even retract) rational. The first such examples over number fields $k$ were given by 
R.~Swan~\cite{swan1969invariant} and V.~Voskresenski{\u\i}~\cite{voskresenskii1970question} and over $k = \bbC$ 
by D.~Saltman~\cite{saltman}. For many specific finite groups $G$ the Noether Problem remains open.

In the case where $G$ is a connected split semisimple groups over $k$, no counter-examples to the Noether Problem are known.
It is known that $BG$ is stably rational for some $G$ (e.g., for $G = \on{GL}_n$, $\on{SL}_n$, $\on{SO}_n$) but for many other connected 
split semisimple groups the Noether Problem remains open.
Among these, projective linear groups $\on{PGL}_n$ and spinor groups
$\on{Spin}_n$ have received the most attention.

For $G = \on{PGL}_n$ the Noether Problem arose independently in ring theory in connection with 
universal division algebras; see~\cite[p.~254]{procesi}. It is known that $BG$ 
is stably rational for every $n$ dividing $420$; the remaining cases are open. See~\cite{lb-rationality} for an overview.

In~\cite{bogomolov1986stable}, F.~A.~Bogomolov claimed that $BG$ is stably rational for every simply connected simple complex
algebraic group $G$. However, there is a mistake in his argument; in particular, it breaks down for the spinor groups.
V.~Kordonski{\u\i}~\cite{kordonskii2000stable} subsequently proved that $B\on{Spin}_7$ and $B\on{Spin}_{10}$ are stably 
rational (again, over the field of complex numbers). More recently,
 Merkurjev~\cite[Section 4]{merkurjev2019rationality} showed that $B\on{Spin}_n$ is retract rational 
 for $n\leq 14$ over any field $k$ of characteristic $\neq 2$\footnote{Over a field of characteristic $0$ retract rationality of $\on{Spin}_{n}$ for $n \leqslant 12$
 was proved earlier by J.-L.~Colliot-Th\'{e}l\`ene and J.-J.~Sansuc~\cite{colliot2007rationality}.},
 and conjectured that $B\on{Spin}_n$ is, in fact, 
 {\em not} retract rational for $n\geq 15$.
We strengthen Merkurjev's result as follows.

\begin{thm} \label{mainthm}
Let $k$ be a field of characteristic $\neq 2$. Then
$B\on{Spin}_{n}$ is stably rational over $k$ for every $n\leq 14$.
\end{thm}

The remainder of this paper will be devoted to proving Theorem~\ref{mainthm}. Our strategy will be as follows.  For $n \leqslant 6$,
it is well known that $\on{Spin}_n$ is special; see~\cite[Section 16.1]{garibaldi2009cohomological}. Hence, $B\on{Spin}_n$ is stably rational; see \Cref{lem.special}.
Merkurjev~\cite[Corollary 5.7]{merkurjev2019rationality} showed that $B\on{Spin}_{2m+1}$ is stably birationally 
equivalent to $B\on{Spin}_{2m+2}$ for every $m\geq 0$. Thus in order to prove Theorem~\ref{mainthm} it suffices to show 
that $B\on{Spin}_n$ is stably rational for $n = 7$, $10$, $11$ and $14$. 
We will give self-contained proofs that work over an arbitrary field $k$ of characteristic $\neq 2$ in 
Propositions~\ref{prop.7}, \ref{prop.10},~\ref{prop.11} and~\ref{prop.14}, respectively.
Along the way we will show that $B(\on{SL}_n \rtimes \bbZ/2 \bbZ)$ is stably rational; see~\Cref{sect.sln}. 
In the last section we give an alternative proof of Theorem~\ref{mainthm} over $k = \mathbb C$
suggested to us by G.~Schwarz. 

\section{Rosenlicht quotients}
\label{sect.rosenlicht}

For the remainder of this paper $k$ will denote a field of arbitrary characteristic and $G$ will denote a smooth linear algebraic group defined over $k$.
Starting from Section~\ref{sect7-10}, we will assume that $\Char(k) \neq 2$, but for now $k$ is arbitrary.

Let $X$ be a reduced and absolutely irreducible algebraic variety equipped 
with an action of $G$ over $k$. A rational map $\pi \colon X \dasharrow Y$ is called a Rosenlicht quotient map for the $G$-action on $X$ if

\begin{itemize}
\item $\pi \circ g = \pi$ for every $g \in G(\cl{k})$,
\item $Y$ is reduced and irreducible, 
\item $k(Y) = k(X)^{G}$ and 
\item $\pi$ is induced by the inclusion of fields $k(X)^{G} \hookrightarrow k(X)$. 
\end{itemize}

It is clear from this definition that a Rosenlicht quotient map exists for every $G$-action: just let
$Y$  be any variety with function field $k(Y) = k(X)^G$, and $\pi \colon X \dasharrow Y$ be the rational 
map induced by the inclusion $k(X)^{G} \hookrightarrow k(X)$.
Note that $Y$ (or $\pi$) is often called {\em the rational quotient} in the literature (see e.g. \cite[2.4]{popov1994}); we will refer to it as
``the Rosenlicht quotient" in this paper in order not to overuse the term ``rational". 

If $\pi \colon X \dasharrow Y$ is a Rosenlicht quotient map, then by a theorem of Rosenlicht \cite[Theorem 2]{rosenlicht} there exist a $G$-invariant 
open $k$-subvariety $X_0 \subset X$ and an open $k$-subvariety $Y_0 \subset Y$ 
such that $\pi$ restricts to a regular map $\pi_0 \colon X_0 \to Y_0$ and 
\begin{equation} \label{e.rosenlicht}
\text{for any $x \in X_0(\overline{k})$, the fiber $\pi_0^{-1}(\pi_0(x))$ equals the $G(\cl{k})$-orbit of $x$.}
\end{equation}
For a modern proof of Rosenlicht's theorem, see~\cite[Section 7]{bdr}. The following Lemma is readily deduced from Rosenlicht's Theorem.

\begin{lem} \label{lem.rosenlicht}
Consider an action of $G$ on $X$ as above, and
let $\pi \colon X \to Z$ be a $G$-equivariant morphism, where $G$ acts trivially on $Z$. Suppose that

(i) there exists a dense open subvariety $Z_{0} \subset Z$ such that for
any $z \in Z_{0}(\overline{k})$ the fiber $\pi ^{-1}(z)=G \cdot x$ for
some $x \in X(\overline{k})$, and

(ii) $\pi $ is generically smooth (which is automatic in characteristic
$0$).

Then $\pi $ is a Rosenlicht quotient for the $G$-action on $X$. In
particular, $\pi $ induces an isomorphism between $k(Z)$ and
$k(X)^{G}$.
\end{lem}

\begin{proof}
Let $\mu \colon X \dasharrow Y$ be a Rosenlicht quotient map. Since $\pi$ is $G$-equivariant, it factors through $\mu $ as follows:
\[ \xymatrix{  X \ar@{-->}[d]_{\mu} \ar@{->}[rd]^{\pi} &  \\
               Y  \ar@{-->}[r]^{ \alpha \quad }                       &  Z.}  \]
(This is called the universal property of Rosenlicht quotients.) It now suffices to show that $\alpha$ is 
a birational isomorphism. Choose open subvarieties $X_0 \subset X$ and $Y_0 \subset Y$ as in Rosenlicht's theorem~\eqref{e.rosenlicht} and such that $\alpha$ is regular on $Y_0$. Then (i) tells us that $\alpha$ induces a bijection between $Y_0(\overline{k})$ and 
$Z_0(\overline{k})$ for some dense open subvariety 
$Z_0 \subset Z$. 
In characteristic $0$ this implies that $\alpha$ is a birational isomorphism, and we are done.

In finite characteristic the fact that $\alpha$ induces a bijection between $Y_0(\overline{k})$ and 
$Z_0(\overline{k})$ only tells us that the field extension $k(Y)/k(Z)$ induced by $\alpha$ is purely inseparable, so we need to use condition (ii) to finish the proof. 
By (ii), $\pi$ is generically smooth and hence, so is $\alpha$. This implies that $\alpha$ is a birational isomorphism, as desired.
\end{proof}

\begin{rmk}\label{rmk.rosenlicht}
Consider a generically free action of $G$ on a variety $X$ defined over $k$. By \cite[Theorem 4.7]{berhuy2003essential}, there exists a dense $G$-invariant open subvariety $X_0 \subset X$ which is the total space of 
a $G$-torsor $\pi \colon X_0 \to Y$ over some $k$-variety $Y$. Conditions (i) and (ii) of~\Cref{lem.rosenlicht} are satisfied, because $G$ is smooth and $\pi$ is a $G$-torsor. It follows that $\pi \colon X_0 \to Y$ (viewed as a rational 
map $X \dasharrow Y$) is a Rosenlicht quotient for the $G$-action on $X$.
\end{rmk}

\section{Preliminaries on the Noether Problem}
\label{prelim}

In this section we collect several known results on the Noether Problem for future use. Given two $k$-varieties, $X$ and $Y$, we will write 
$X\sim Y$ if $X$ and $Y$ are birationally isomorphic over $k$.

Recall that a smooth linear algebraic group $G$
is called {\em special} if $H^1(K, G) = \{ 1 \}$ for every field $K$ 
containing $k$. Special groups were introduced by Serre~\cite{serre1958espaces}; over an algebraically closed field of characteristic $0$
they were classified by Grothendieck~\cite{grothendieck1958torsion}. 

\begin{lemma} \label{lem.special} 
If $G$ is special and stably rational, then
$BG$ is stably rational.
\end{lemma}

\begin{proof} See \cite[Proposition 4.7]{colliot2007rationality} or \cite[Remark 3.2]{florence2018genus0}.
\end{proof}

\begin{example} \label{ex.special}
It is known that the groups $G = \on{GL}_n, \on{SL}_n, \on{Sp}_{2n}$ are special for every $n \geqslant 1$,
and so are the groups $G = \on{Spin}_n$ for $n \leqslant 6$. Thus $BG$ is stably rational for these $G$.
\end{example}

Let $P \to \Sym_n$ be a permutation representation of a linear algebraic group $P$.
If $G$ is a another linear algebraic group, this representation gives rise to the wreath product
$G \wr P$, which is defined as the semidirect product $G^n\rtimes P$ via the permutation action of $P$ on $G^n$.

\begin{lemma} \label{birational} Let $G$, $G_1$, $H$, and $P$ be linear algebraic groups over $k$, and $P \to \Sym_n$ be a permutation representation.

\smallskip
(a)	If $BG$ and $BH$ are stably birational, then $B(G \wr P)$ and $B(H \wr P)$ are stably birational. 

\smallskip
(b) If $BG$ is stably rational, then $B(G \wr P)$ is stably birational to $BP$.

\smallskip
(c) If $BG$ is stably rational, then $B(G \times G_1)$ is stably birational to $B(G_1)$. 
\end{lemma}

\begin{proof}
(a)	Let $V$ be a generically free $G$-representation, and $W$ be a generically free $H$-representation. By our
assumption the Rosenlicht quotients $V/G$ and $W/H$ are stably birationally equivalent, say $V/G\times \A^r\sim W/H\times \A^s$. After replacing $V$ by $V\oplus \A^r$ and $W$ by $W\oplus \A^s$, where $G$ acts trivially on $\A^r$ and $H$ acts trivially on $\A^s$, we may assume that $V/G\sim W/H$.

The product actions of $G^n$ on $V^n$ and of $H^n$ on $W^n$ naturally extend to linear representations 
\[ G \wr P \to \GL(V^n) \quad \text{and} \quad H \wr P \to \GL(W^n), \]
respectively, where $P$ acts on $V^n$ and $W^n$ by permuting the factors. Now let $P \to \GL(Z)$ be some generically free linear representation of $P$. 
Then the representations $G^n \rtimes P$ on $V^n \times Z$ and of $H^n \rtimes P$ on $W^n \times Z$ are generically free. 
Comparing the Rosenlicht quotients 
$(V^n \times Z)/(G \wr P)$ and $(W^n \rtimes Z)/(H \wr P)$,
we obtain
\[ (V^n \times Z)/(G \wr P)\sim ((V/G)^n \times Z)/ P \sim  ((W/H)^n \times Z)/P \sim (W^n \rtimes Z)/(H \wr P) , \]
as desired.
	
(b)	Letting $H$ be the trivial group in part (a), we deduce that $B(G^n\rtimes P)$ is stably birational to $BP$.

(c) is a special case of (b) with $P = G_1$, equipped with the trivial permutation representation $P \to \Sym_1$. 
\end{proof}

    \begin{lemma}\label{stabilizer}
	Let $G \to \GL(V)$ be a finite-dimensional representation defined over $k$. Suppose there exists a $k$-point $v_0\in V$ such that the scheme-theoretic stabilizer 
	$H$ of $v_0$ in $G$ is smooth, and the $G$-orbit of $v_0$ is dense in $V$. Then $BG$ and $BH$ are stably birationally equivalent.	
	\end{lemma}

	When $G$ is reductive and $\on{char}k=0$, this lemma reduces to~\cite[Proposition 3.13]{colliot2007rationality}.  
	The following argument works in arbitrary characteristic.  
	
		\begin{proof} Let $G \to \GL(W)$ be a generically free representation of $G$. Denote the open orbit of $v_0$ in $V$ by $V_0$ and
    the Rosenlicht quotient map for the $H$-action on $W$ by $\pi \colon W \dasharrow W/H$. 
	After possibly replacing $W/H$ by a dense open subvariety, we can choose an $H$-invariant dense open subvariety 
	$W_0 \subset W$ such that $\pi$ restricts to a morphism $W_0 \to W/H$ whose fibers are exactly the $G$-orbits in $W_0$; see (\ref{e.rosenlicht}). In fact, by \Cref{rmk.rosenlicht}, we may assume that $\pi \colon W_0 \to W/H$ is an $H$-torsor.
	
	We claim that $\phi \colon V_0 \times W_0 \to W/H$ given by $(v, w) \to \pi (w)$ is a Rosenlicht quotient map for
	the $G$-action on $V_0 \times W_0$. If we establish this claim, then 
	\[ k(V \times W)^G = k(V_0 \times W_0)^G = k(W/H) = k(W)^H, \]
	and the lemma will follow. By~\Cref{lem.rosenlicht}, it suffices to show that

    \smallskip
    (i) $\phi^{-1}(\pi(w))$ is a single $G$-orbit for any $w \in W_0(\overline{k})$, and
    
    \smallskip
    (ii) $\phi$ is generically smooth.
    
    \smallskip
    \noindent
    To prove (i), suppose $\phi(v_1, w_1) = \phi(v_2, w_2)$ for some $(v_1, w_1)$ and $(v_2, w_2)$ in $V_0 \times W_0$.
    Our goal is to show that $(v_1, w_1)$ and $(v_2, w_2)$ lie in the same $G$-orbit. After translating these points by suitable elements of $G$, we may assume that $v_1 = v_2 = v_0$. Since $\pi$ is an $H$-torsor and
    \[ \pi(w_1) = \phi(v_0, w_1) = \phi(v_0, w_2) = \pi(w_2), \]
    we conclude that $w_2 = h(w_1)$ for some $h \in H(\cl{k})$. Since $v_0$ is stabilized by $H$, 
    we have $h(v_0, w_1) = (v_0, w_2)$, as desired.
    
    To prove (ii), note that $\pi$ is the composition of the projection map $p \colon V_0 \times W_0 \to W_0$ 
    and the Rosenlicht quotient map $\pi \colon W_0 \to W/H$. Clearly, $p$ is smooth. Moreover, $\pi$ is also smooth, because $H$ is smooth and $\pi$ is an $H$-torsor. Thus $\phi = \pi \circ p$ is smooth, as desired.
    This completes the proof of (ii) and thus of \Cref{stabilizer}.
\end{proof}

\begin{example} \label{ex.mu} 
Consider the 1-dimensional representation $V$ of $G = \bbG_m$, where $t \in \bbG_m$ acts on $V$ via scalar multiplication by $t^n$, where $n$ 
is not divisible by $\Char(k)$. Taking $v_0$ to be any non-zero vector in $V$, we see that the stabilizer of $v_0$ in $G = \G_m$ is $H = \mu_n$. 
Since $\bbG_m = \on{GL}_1$ is special, $B \bbG_m$ is stably rational over $k$. By~\Cref{stabilizer}, so
is $B \mu_n$. 
\end{example}

For more sophisticated applications of \Cref{stabilizer}, see~\cite[Section 4]{colliot2007rationality}.

\begin{rmk} 
Representations of connected groups which admit a dense open orbit have been studied by M.~Sato and T.~Kimura \cite{sato1977classification} (over $\C$). 
They referred to such representations as prehomogeneous vector spaces. 
\end{rmk}

\begin{lemma}\label{affine} {\rm(}cf.~\cite[Corollary to Lemma 2.2]{bogomolov1986stable}{\rm )}
    Let $V$ be a linear representation of $G$, and consider $V$ as a vector group scheme over $k$. Then $B(V\rtimes G)$ is stably birational to $BG$.
\end{lemma}
	
\begin{proof}
    Let $W$ be a generically free representation of $G$, and consider the $G$-representation $V_0:=\A^1\oplus V\oplus W$, where $G$ acts trivially on $\A^1$. We let the vector group scheme $V$ act linearly on $\A^1\oplus V$ by \[v\cdot (\lambda,v'):=(\lambda,\lambda v+v')\] and trivially on $W$. This gives $V_0$ the structure of a $V\rtimes G$-representation. Since $G$ acts generically freely on $W$ and $V$ acts generically freely on $\A^1\oplus V$, $V_0$ is generically free as a $V\rtimes G$-representation. It suffices to show that $V_0/(V\rtimes G)$ is stably birational to $W/G$.
    
    The projection map $\pi:V_0\to \A^1\oplus W$ is $V\rtimes G$-equivariant. Moreover, $V$ acts trivially on $\A^1\oplus W$ and simply transitively on the fibers of points in $(\A^1\setminus\set{0})\times W$. By \Cref{lem.rosenlicht}, $\pi$ is the Rosenlicht quotient map for the $V$-action on $V_0$. Hence
    \[V_0/(V\rtimes G)\sim (\A^1\oplus W)/G \sim \A^1 \times W/G,\] as desired.
\end{proof}	

\section{The Noether Problem for $\SL_n \rtimes \, (\bbZ/{2}\bbZ)$} 
\label{sect.sln}

Let $\mathbb{Z}/2 \mathbb Z = \langle \tau \rangle$ be the cyclic group of order $2$.
In this section we will study the Noether Problem for the group
$\SL_n \rtimes (\mathbb Z/ 2 \bbZ)$, where $n \geqslant 1$ and $\tau$ acts on $\SL_n$ by $A \to (A^{-1})^T$.
Our main result is as follows.

\begin{prop} \label{sln}
$B(\SL_n \rtimes (\mathbb Z/ 2 \bbZ))$ is stably rational for every $n \geqslant 1$.
\end{prop}

In the sequel we will write $\Mat_{a \times b}$ for the space of rectangular matrices with $a$ rows and $b$ columns. 
Assume $n \geqslant 2$ and consider the linear representation of $\SL_n \rtimes (\mathbb Z/ 2 \bbZ)$ on 
$V = \Mat_{n \times (n-1)} \times \Mat_{(n-1) \times n}$ 
given by 
\[ \text{$A \colon (X, Y) \mapsto (AX, Y A^{-1})$ for any $A \in \SL_n$
and $\tau \colon (X, Y) \to (Y^T, X^T)$.} \]
Here $X^T$ denotes the transpose of $X$ and similarly for $Y$.
This action is well defined:
\[ \tau (A \cdot (X, Y)) = \big( (A^{-1})^T Y^T, X^T A^T \big) = (A^{-1})^T\cdot \tau(X, Y)  \]
for every $A \in \on{SL}_n$, $X \in \Mat_{n \times (n-1)}$, and $Y \in \Mat_{(n-1) \times n}$. Set
\[ \pi \colon V \to \Mat_{(n-1) \times (n-1)}, \]
where $\pi (X, Y) = YX$. Clearly 
\begin{equation} \label{e.invariant}
\pi(A \cdot(X, Y)) = \pi(X, Y) 
\end{equation}
for any $A \in \on{SL}_n$, $X \in \Mat_{n \times (n-1)}$ and $Y \in \Mat_{(n-1) \times n}$.

\begin{lem} \label{sln1} Suppose $\pi(X, Y)$ is a non-singular $(n-1) \times (n-1)$ matrix for some
$n \geqslant 2$, $X \in \Mat_{n \times (n-1)}(k)$ and $Y\in \Mat_{(n - 1) \times n}(k)$. Then

\smallskip
(a) the $\SL_n$-orbit of $(X, Y)$ in $V$ contains a point of the form $(J, Y')$, where
\[ J = \begin{pmatrix} 
1 & 0 & \ldots & 0 \\
0 & 1 & \ldots & 0 \\
\hdotsfor{4}   \\
0 & 0 & \ldots & 1 \\
0 & 0 & \ldots & 0 \end{pmatrix}  \in \Mat_{n \times (n-1)} \quad \text{and} \quad Y' \in \Mat_{(n-1) \times n}. \]

\smallskip
(b) The scheme-theoretic stabilizer $\Stab_{\SL_n}(X, Y)$ is trivial and
$\pi^{-1} (\pi (X, Y))$ is a single $\on{SL}_n$-orbit.

\smallskip
(c) $\pi$ is a Rosenlicht quotient map for the $\SL_n$-action on $V$. In particular,
$\pi$ induces an isomorphism between $k(\Mat_{(n-1) \times (n-1)})$ and $k(V)^{\SL_n}$.
\end{lem}

\begin{proof} 
(a) is a consequence of the fact that $\SL_n$ acts transitively on $(n-1)$-tuples of linearly independent vectors 
in $k^n$. 

\smallskip
(b) In view of part (a), we may assume that $X = J$. If 
\begin{equation} \label{e.Y} 
Y  = \begin{pmatrix}
y_{1,1} & y_{1,2} & \ldots & y_{1,n -1} & y_{1, n} \\
y_{2,1} & y_{2,2} & \ldots & y_{2,n-1} & y_{2, n} \\
\hdotsfor{5}   \\
y_{n-1, 1}  & y_{n-1, 2} & \ldots & y_{n-1, n} & y_{n-1, n} \\ 
\end{pmatrix},
\end{equation}
then
\begin{equation} \label{e.piJY} \pi(J, Y) = Y J = \begin{pmatrix} 
y_{1,1} & y_{1,2} & \ldots & y_{1,n -1}  \\
y_{2,1} & y_{2,2} & \ldots & y_{2,n-1} \\
\hdotsfor{4}   \\
y_{n-1, 1}  & y_{n-1, 2} & \ldots & y_{n-1, n-1} \\ 
\end{pmatrix}. \end{equation} 
Now suppose $(X', Z)$ is another point in the fiber $\pi^{-1} (\pi(J, Y))$.
We want to show that $(X', Z)$ is an $\SL_n$-translate of $(J, Y)$.
By part (a), we may assume without loss of generality that $X' = J$.
Since we are assuming that $\pi(J, Y) = \pi(J, Z)$, this tells us that
\begin{equation} \label{e.Z} Z  = \begin{pmatrix}
y_{1,1} & y_{1,2} & \ldots & y_{1,n -1} & z_{1, n} \\
y_{2,1} & y_{2,2} & \ldots & y_{2,n-1} & z_{2, n} \\
\hdotsfor{5}   \\
y_{n-1, 1}  & y_{n-1, 2} & \ldots & y_{n-1, n} & z_{n-1, n} \\ 
\end{pmatrix}
\end{equation}
for some $z_{1, n}, \ldots, z_{n-1, n} \in k$.

We claim that the locus $\Lambda$ of solutions to the system
$A \cdot (J, Y) = (J, Z)$,
or equivalently 
\[ \left\{ \begin{array}{l}  
AJ=J \\
ZA=Y,\\
\end{array} \right. \]
is (scheme-theoretically) a single point $A \in \SL_n$. The fact that $\Lambda$ is non-empty
implies that $(J, Y)$ and $(J, Z)$ lie in the same $\SL_n$-orbit. The fact that $\Lambda$ is a single point tells us
that the scheme-theoretic stabilizer of $(J, Y)$ is trivial (just set $Z = Y$ in the claim).

It thus remains to prove the claim. Note that $AJ = J$ if and only if
\begin{equation} \label{e.A} A  = \begin{pmatrix} 
1 & 0 & \ldots & 0 & t_1 \\
0 & 1 & \ldots & 0 & t_2 \\
\hdotsfor{5}   \\
0 & 0 & \ldots & 1 & t_{n-1} \\
0 & 0 & \ldots & 0 & 1  \end{pmatrix} \end{equation}
for some $(t_1, \dots, t_{n-1}) \in \bbA^{n-1}$.
On the other hand, $Z A = Y$ translates to
\[ \left\{ \begin{array}{l}  
y_{1, 1} t_1 + \dots + y_{1, n-1} t_{n-1} + z_{1, n} = y_{1, n} \\
y_{2, 1} t_1 + \dots + y_{2, n-1} t_{n-1} + z_{2, n} = y_{2, n}\\
\hdotsfor{1} \\
y_{n - 1, 1} t_1 + \dots + y_{n-1, n-1} t_{n-1} + z_{n-1, n} = y_{n-1, n} . \end{array} \right. \]
The matrix of this linear system is~\eqref{e.piJY}. This matrix is non-singular by our assumption.
Hence, the system has a unique solution, $(t_1, \dots, t_{n-1})$ in $\bbA^{n-1}$. This completes the proof of the claim and thus of part (b).

\smallskip
(c) In view of~\eqref{e.invariant} and part (b), it suffices to show that
 $\pi$ is generically smooth; see \Cref{lem.rosenlicht}. In other words, we need to check that the differential $d\pi \colon T_v(V) \to
T_\pi (M_{(n-1) \times (n-1)})$ is surjective for $v \in V$ in general position. This is readily seen 
by restricting $\pi$ 
to the affine subspace $\{ J \} \times \Mat_{(n-1) \times n}$ and using the formula~\eqref{e.piJY}.
\end{proof}

\begin{proof}[Proof of \Cref{sln}] First let us settle the case, where $n = 1$. 
Here $\SL_n \rtimes (\bbZ/ 2 \bbZ) \simeq \bbZ/ 2 \bbZ$.
Examining the natural two-dimensional permutation representation $W$ of $\bbZ/ 2 \bbZ \simeq \Sym_2$, 
we readily see that $k(W)^{\bbZ/ 2 \bbZ}$ is
rational over $k$. (Here $k$ is a field of arbitrary characteristic.) Consequently,
\begin{equation} \label{e.n=1} 
\text{$B(\bbZ/ 2 \bbZ)$ is stably rational over $k$.}
\end{equation}

Now suppose $n \geqslant 2$. Set $V = \Mat_{n \times (n-1)} \times \Mat_{(n-1) \times n}$ and consider the representation $V\times W$ of $\SL_n \rtimes (\bbZ/ 2 \bbZ)$, where $\SL_2 \rtimes (\bbZ/ 2 \bbZ)$ acts on $W$ via $\Z/2\Z$. 
Since the $\SL_n$-action on $V$ is generically free (see~\Cref{sln1}(b)) and 
the $\bbZ/2 \bbZ$-action on $W$ is generically free (obvious), we conclude that so is the $\SL_n \rtimes (\mathbb Z/ 2 \bbZ)$-action on $V\times W$. 
It remains to show that
\begin{equation} \label{e.sln}
\text{ $k(V\times W)^{\SL_n \rtimes (\mathbb Z/ 2 \bbZ)}$ is stably rational over $k$.}
\end{equation}
In view of \Cref{sln1}(c), we have
\[ k(V\times W)^{\SL_n \rtimes (\mathbb Z/ 2 \bbZ)} = k\big( (V/\SL_n) \times W \big)^{\bbZ/2\bbZ} = k(\Mat_{(n-1) \times (n-1)} \times W)^{\bbZ/2 \bbZ}. \]

To see how $\bbZ/ 2 \bbZ = \langle \tau \rangle$ acts on $V/\SL_n = \Mat_{(n-1) \times (n-1)}$, recall that
$\tau$ sends $v = (X, Y) \in V$ to $(Y^T, X^T)$. Hence, the induced
action of $\tau$ on $V/\SL_n = \Mat_{(n-1) \times (n-1)}$ takes $\pi(X, Y) = YX$ to
$\pi(Y^T, X^T) = X^T Y^T = \pi(X, Y)^T$. In other words, the induced $\bbZ/2 \bbZ$-action 
on $V/\SL_n = \Mat_{(n-1) \times (n-1)}$ is given by $\tau \colon Z \mapsto Z^T$.
In particular, this action is linear. 

Now \eqref{e.n=1}, tells us that $k(\Mat_{(n-1) \times (n-1)} \times W)^{\bbZ / 2 \bbZ}$ is stably rational over $k$. 
This completes the proof of~\eqref{e.sln} and thus of~\Cref{sln}.
\end{proof}

\begin{rmk}
For $n \geqslant 3$, the $\SL_n \rtimes (\bbZ/ 2 \bbZ)$-action on $V$ is generically free, so we can work directly with $V$, rather than $V \times W$. 
The extra factor of $W$ is only needed when $n = 2$. Note also that if $\Char(k) \neq 2$, then
$\bbZ/ 2 \bbZ$ is isomorphic to $\mu_2$, and thus~\eqref{e.n=1} is a special case of \Cref{ex.mu}.
\end{rmk}

\section{Group extensions}

In the sequel we will apply \Cref{sln} in combination with the following proposition.

\begin{prop}\label{semidirect}
	Let $n$ be an odd integer and let $d\geq 1$. Consider a short exact sequence 
\[	\xymatrix{  1 \ar@{->}[r] & \SL_n \ar@{->}[r] & H \ar@{->}[r]^{\quad\pi\qquad} & \bbZ / 2 \bbZ \ar@{->}[r] & 1} \]
of algebraic groups. Then either $H \simeq \on{SL}_n\times (\bbZ / 2 \bbZ)$ or $H \simeq \on{SL}_n\rtimes (\bbZ / 2 \bbZ)$, where the generator $\tau$ of $\bbZ / 2 \bbZ$ acts by $\tau \colon A  \mapsto (A^{-1})^{T}$, as in \Cref{sect.sln}.
\end{prop}
	
\begin{proof} Note that the fiber $\pi^{-1}(\tau)$ is an $\SL_n$-torsor.
    Since $\SL_n$ is a special group, this torsor is split. In other words, there exists an $x\in H(k)$ such that $\pi(x)= \tau$. Let $\phi_x:\on{SL}_n\to \on{SL}_n$ denote conjugation by $x$: $\phi_x(A)=xAx^{-1}$. Since $\phi_x$ is a $k$-group automorphism of $\on{SL}_n$, there exists $B\in \on{SL}_n(k)$ such that either $\phi_x(A)=BAB^{-1}$ for every $A\in \on{SL}_n$, or $\phi_x(A)=B(A^{-1})^TB^{-1}$. After replacing $x$ by $Bx$, we may assume that either $\phi_x=\on{Id}$ or $\phi_x(A)=(A^{-1})^T$. It now suffices to show that there exists a $y \in H(k)$ such that $\pi(y)= \tau$ and $y^2=1$. 
    
    In both cases, $\phi_{x^2}=(\phi_x)^2$ equals the identity, i.e., $x^2Ax^{-2}=A$ for every $A\in \on{SL}_n({k})$. It follows $x^2$ lies in the center of $\on{SL}_n$, i.e., $x^2 \in \mu_n(k) \subset \on{SL}_n(k)$ is a  diagonal matrix.
    
    Let $\ang{x}\subset H(k)$ be the subgroup generated by $x$. The restriction of $\pi$ to $\ang{x}$ is surjective and it sends $x^m$ to $\tau^m$. It follows that $\on{Ker}(\pi) \cap \ang{x} =\ang{x^2}$, so that we have a short exact sequence
    \[	\xymatrix{  1 \ar@{->}[r] & \ang{x^2} \ar@{->}[r] & \ang{x} \ar@{->}[r]^{\quad\pi\qquad} & \bbZ / 2 \bbZ \ar@{->}[r] & 1.} \]
    The order of $\mu_n(k)$ divides $n$, hence $\mu_n(k)$ is cyclic of odd order. Since $\ang{x^2}$ is a subgroup of $\mu_n(k)$, it is also cyclic of odd order. On the other hand, since $\ang{x}$ surjects onto $\bbZ/ 2 \bbZ$, $\ang{x}$ is of even order. We conclude that $\ang{x}$ contains an element $y$ of order $2$, and $\pi(y) = \tau$ as desired.
\end{proof}	

\section{The Noether Problem for $G_2$, $\on{Spin}_7$ and $\on{Spin}_{10}$}
\label{sect7-10}

For the remainder of this paper we will assume that $k$ is a field of characteristic $\neq 2$. In particular, 
$\mu_2 \simeq \bbZ / 2 \bbZ$ over $k$.

\begin{prop}\label{prop.g2}
$BG_2$ is stably rational.
\end{prop}

\begin{proof}
    Let $V$ be the octonion representation of $G_2$. If we let $\G_m$ act on $V$ by scaling, the induced $G_2\times \G_m$-action on $V$ has an open orbit, and the stabilizer in general position is isomorphic to $\on{SL}_3\rtimes \mu_2$, where $\mu_2$ acts on $\on{SL}_3$ by $A\mapsto (A^{-1})^T$; see \cite[Theorem 4]{jacobson1958composition}, \cite[Proposition 3.3]{singh2005reality} or \cite[Lemma 3.3]{pirisi2017motivic}.\footnote{In~\cite{pirisi2017motivic} it is assumed that $\sqrt{-1}\in k$. However, this assumption can be dropped if one works with the split form of $G_2$ throughout.} By \Cref{stabilizer}, $BG_2$ is stably birationally equivalent 
    to $B(\on{SL}_3\rtimes \mu_2)$. By \Cref{sln}, $B(\on{SL}_3\rtimes \mu_2)$ is stably rational, and hence so is $BG_2$.
\end{proof}

 For an alternative proof of Proposition~\ref{prop.g2}, see \cite[Corollary 2, p.~568]{zubkov-shestakov}.

\begin{prop}\label{prop.7}
$B\on{Spin}_{7}$ is stably rational.
\end{prop}

\begin{proof}  
    Let $V$ be the spin representation of $\on{Spin}_{7}$. Letting $\G_m$ act on $V$ by scalar multiplication, we obtain an action of $\on{Spin}_{7}\times \G_m$ on $V$. If $\gamma$ is the generator of the center of $\on{Spin}_{7}$, then $\gamma$ acts as $-\on{Id}$ on $V$. It follows that the subgroup \[ C:=\ang{(\gamma,-1)}\simeq \mu_2 \] of $\on{Spin}_{7}\times \G_m$ acts trivially on $V$. The quotient $(\on{Spin}_{7}\times \G_m)/C$ acts faithfully on $V$. This quotient group is usually called
    ``the even Clifford group" and is denoted by $\Gamma_{7}^+$. 
    The natural short exact sequence
    \[	\xymatrix{  1 \ar@{->}[r] & \on{Spin}_7 \ar@{->}[r] & \Gamma_7^+ \ar@{->}[r] & \G_m \ar@{->}[r] & 1} \]
	shows that $B\on{Spin}_{7}$ and $B\Gamma_{7}^+$ are stably birational; see \cite[Lemma 4.3]{merkurjev2019rationality} or~\cite[Corollary 4.4]{merkurjev2019rationality}.
		
		It remains to show that $B\Gamma_{7}^+$ is stably rational.
     By \cite[Proposition 4]{igusa1970classification}, there exists a quadratic form $g:V\to \A^1$ such that the orbits of the $\on{Spin}_7$-action on $V$ are exactly the fibers of $g$. In particular, the $\Gamma_7^+$-action on $V$ has an open orbit. Furthermore, if $p\in V(k)$, $g(p) \neq 0$, and $S$ is the stabilizer of $p$ in $\on{Spin}_7$,  then
\begin{equation} \label{e.7} 
\text{$S \simeq G_2$ as group schemes;} 
\end{equation}
     see also~\cite[Table 1]{garibaldi2017spinors}. Note that Igusa only showed that $S \simeq G_2$ at the level of points.  This settles~\eqref{e.7} in characteristic $0$. To complete the proof of~\eqref{e.7} in finite characteristic ($\neq 2$) it remains to show 
     that $S$ is smooth, or equivalently, that $\on{dim}\on{Lie}(S)=\dim S$, where $\on{Lie}(S)$ denotes the Lie algebra of $S$. This is done in~\cite[pp. 115-116]{sato1977classification}, where it is shown that~$\on{Lie}(S)\simeq\mathfrak{g}_2$;
     see~\Cref{rem.char2}.~\footnote{This argument is implicit in~\cite{garibaldi2017spinors};
     we are grateful to S.~Garibaldi for clarifying it to us.}
     
     Denote by $H$ the stabilizer of $p$ in $\on{Spin}_7\times\G_m$.
    If $(h,t)\in {H}(\cl{k})$, then \[g(p)=g((h,\lambda)p)=g(\lambda hp)=\lambda^2g(hp)=\lambda^2g(p).\] Since $g(p)\neq 0$, we deduce that $\lambda^2=1$. Thus the projection $\on{Spin}_7 \times \bbG_m \to \bbG_m$
    to the second coordinate gives rise to a short exact sequence 
    \[	\xymatrix{  1 \ar@{->}[r] & S \ar@{->}[r] & H \ar@{->}[r]^{\pi} & \mu_2 \ar@{->}[r] & 1.} \]
     Since $S \simeq G_2$ has trivial center, $C$ intersects $S$ trivially. The inclusion $S \hookrightarrow {H}$ now
     induces an isomorphism between $S$ and $H/C$, which is the stabilizer of $p$ in $\Gamma_7^+$. 
     By \Cref{stabilizer}, $B\Gamma_7^+$ is stably birational to $B(H/C) = BS = BG_2$. By \Cref{prop.g2}, $BG_2$ is stably rational; hence so is $B\Gamma_7^+$.
\end{proof}

\begin{rmk} \label{rem.char2} 
The base field in~\cite{sato1977classification} is assumed to be $\C$. However, the Lie algebra calculation of \cite[pp. 115-116]{sato1977classification} remains valid over any field $k$ of characteristic $\neq 2$. The same is true for the Lie algebra calculations from \cite{sato1977classification} which will be used in the proofs 
of Propositions~\ref{prop.10} and~\ref{prop.11}.
\end{rmk}


\begin{prop}\label{prop.10}
$B\on{Spin}_{10}$ is stably rational.
\end{prop}

\begin{proof}
    Let $V$ be the half-spin representation of $\on{Spin}_{10}$. 
    By \cite[Proposition 2]{igusa1970classification} there are two non-zero orbits in $V$. Let $H$ be the stabilizer of a $k$-point in the open orbit. The subgroup $H$ is explicitly described in \cite[Lemma 3]{igusa1970classification} (with $n=5$): we have $H=W\rtimes G_0$, where $W$ has the structure of an $8$-dimensional vector space and $G_0$ acts linearly on $W$. By~\cite[Proposition 1]{igusa1970classification}, $G_0\simeq\on{Spin}_7$.~\footnote{Here we could have cited~\cite[Proposition 2]{igusa1970classification} which asserts that $H=(\on{Spin}_7)\cdot (\G_a)^8$. In \cite{igusa1970classification}, 
    the symbol $H=H_1\cdot H_2$ is used for the semidirect product $H=H_2\rtimes H_1$; see e.g.~\cite[Lemma 1]{igusa1970classification}. To avoid confusion, we spelled out the argument in detail.} Note that $H$ is smooth by the Lie algebra computation in~\cite[p.~121]{sato1977classification}; once again, see~\Cref{rem.char2}. 
    Hence $H\simeq W\rtimes \on{Spin}_7$ as group schemes (and not just at the level of points).
    
    By~\Cref{stabilizer}, $B\on{Spin}_{10}$ is stably birational to $BH=B(W\rtimes\on{Spin}_7)$. By \Cref{affine}, $B(W\rtimes \on{Spin}_7)$ is stably birational to $B\on{Spin}_7$, which is stably rational by \Cref{prop.7}. We conclude that $B\on{Spin}_{10}$ is stably rational.
\end{proof}
	
\section{The Noether Problem for $\on{Spin}_{11}$}
\label{sect11}

\begin{prop}\label{prop.11}
$B\on{Spin}_{11}$ is stably rational.
\end{prop}

Our proof will follow the same pattern as the proof of stable rationality of $B\on{Spin}_7$ in \Cref{prop.7}, except that 
the second half of the argument will be more involved. 
		
	\begin{proof}[Proof of \Cref{prop.11}] Let $V$ be the spin representation of $\on{Spin}_{11}$. 
	Our starting point is the following result of J.~Igusa~\cite[Proposition 6]{igusa1970classification}.

\smallskip
		(a) There exists a non-zero $\on{Spin}_{11}$-invariant homogeneous form $J:V\to \A^1$ of degree $4$, such that the $\on{Spin}_{11}$-orbits in $V\setminus\set{0}$ are the $J^{-1}(\lambda)$, $\lambda\in\A^1\setminus\set{0}$, together with four orbits inside $J^{-1}(0)$.
		
\smallskip
        (b) If $-\lambda\in \A^1(k)$ is a non-zero square, the orbit $J^{-1}(\lambda)$ contains a rational point whose stabilizer is $k$-isomorphic to $\on{SL}_5$. This is an isomorphism of group schemes; see the proof of~\cite[Proposition 39]{sato1977classification} or \cite[Table 1]{garibaldi2017spinors}.

	\smallskip
		Letting $\G_m$ act on $V$ by scalar multiplication, we obtain an action of $\on{Spin}_{11}\times \G_m$ on $V$. If $\gamma$ is the generator of the center of $\on{Spin}_{11}$, then $\gamma$ acts as $-\on{Id}$ on $V$. It follows that the subgroup $C:=\ang{(\gamma,-1)}\simeq \mu_2$ of $\on{Spin}_{11}\times \G_m$ acts trivially on $V$. The quotient
		$(\on{Spin}_{11}\times \G_m)/C$ acts faithfully on $V$. This quotient is usually called the even Clifford group and is denoted by
		$\Gamma_{11}^+$. Applying \cite[Lemma 4.3]{merkurjev2019rationality} to the short exact sequence
		 \[	\xymatrix{  1 \ar@{->}[r] & \on{Spin}_{11} \ar@{->}[r] & \Gamma_{11}^+ \ar@{->}[r]^{\rm pr_2} & \G_m \ar@{->}[r] & 1} \]
		we see that $B\on{Spin}_{11}$ and $B\Gamma_{11}^+$ are stably birational; cf.~\cite[Corollary 4.4]{merkurjev2019rationality}.
		It remains to show that $B\Gamma_{11}^+$ is stably birational.
	   
		By (a), $\on{Spin}_{11}\times \G_m$ acts transitively on $U:=V\setminus J^{-1}(0)$. Let $p\in U(k)$, $H$ be the stabilizer of $p$ in $\on{Spin}_{11}\times \G_m$, and $\overline{H}$ be the stabilizer of $p$ in $\Gamma_{11}^+ = (\on{Spin}_{11} \times \G_m)/C$. Since $C$ acts trivially on $V$, 
		$C$ is a central subgroup of $H$ and $\overline{H} = H/C$. Since $p$ belongs to the open orbit of the $\Gamma_{11}^+$-action, \Cref{stabilizer} tells us that
		$B\Gamma_{11}^+$ is stably birationally equivalent to $B\overline{H}$. Thus it suffices to show that $B \overline{H}$ is stably rational.
		
		If $(h,\lambda)\in H(\cl{k})$, then 
		\[J(p)=g((h,\lambda)p)=J(\lambda hp)=\lambda^4J(hp)=\lambda^4J(p).\] 
		Since $J(p)\neq 0$, we deduce that $\lambda^4=1$. This shows that $H \c \on{Spin}_{11}\times \mu_4$. 
		The kernel of the projection 
		$\pi: H\to \mu_4$ is the stabilizer of $p$ in $\on{Spin}_{11}\times\set{1}$. By (b),
		this stabilizer is isomorphic to $\on{SL}_5$. We thus obtain a short exact sequence
		\[	\xymatrix{  1 \ar@{->}[r] & \SL_5 \ar@{->}[r] & \overline{H} \ar@{->}[r]^{\pi} & \mu \ar@{->}[r] & 1,} \]
		where $\mu:=\on{Im}(\pi)$.
		 Note that $C \cap \on{Spin}_{11} = 1$ and thus $C \cap \on{SL}_5 = 1$. Modding out by $C$, we obtain a short exact sequence 
		\[	\xymatrix{  1 \ar@{->}[r] & \SL_5 \ar@{->}[r] & \overline{H} \ar@{->}[r] & \mu/\pi(C) \ar@{->}[r] & 1.} \]
		Since $C \cap \on{SL}_5 = 1$, we have $\pi(C)\simeq \mu_2$. If $\mu=\pi(C)$, then $\overline{H} \simeq \on{SL}_5$ is special. 
		In this case $B\overline{H}$ is stably rational by~\Cref{lem.special}. Hence we may assume 
		that $\mu = \mu_4$. In this case $\mu/\pi(C) \simeq \mu_2$ and our exact sequence reduces to
		\[	\xymatrix{  1 \ar@{->}[r] & \SL_5 \ar@{->}[r] & \overline{H} \ar@{->}[r] & \mu_2 \ar@{->}[r] & 1.} \]
		By \Cref{semidirect} 
		 either (i) $\overline{H}=\on{SL}_5\times \mu_2$ or (ii) $\overline{H}=\on{SL}_5\rtimes \mu_2$, where $\mu_2$ acts on $\on{SL}_5$ by $A\mapsto (A^{-1})^T$.
		 
		In case (i), $B\overline{H}$ is stably birational to $B \mu_2$ by \Cref{birational}(c), and $B\mu_2$ is stably rational by Example~\ref{ex.mu}. Thus $B\overline{H}$ is stably rational. In case (ii), $B\overline{H}$ is stably rational by~\Cref{sln}.
	\end{proof}

\section{The Noether Problem for $\on{Spin}_{14}$}
\label{sect14}

\begin{prop}\label{prop.14}
$B\on{Spin}_{14}$ is stably rational.
\end{prop}

\begin{proof} 
	Let $V$ be the half-spin representation of $\on{Spin}_{14}$, $v\in V$ be a $k$-point in general position, $S$ be the stabilizer of $v$, and $N$ be the normalizer of $S$. By~\cite[Example 21.1]{garibaldi2009cohomological}, 
	\[ N\simeq (G_2\times G_2)\rtimes \mu_8, \; \quad G_2\times G_2\subseteq S \subseteq N, \] and 
	the $\on{Spin}_{14}$-orbit of $[v]$ is open in $\P(V)$; cf.~also~\cite[\S 8]{garibaldi2017spinors}.
	The $\mu_8$-action on $G_2\times G_2$ factors through the surjection $\mu_8\to \mu_2$, where $\mu_2\simeq S_2$ acts on $G_2\times G_2$ by switching the two factors. Note that this action is well defined even if $k$ does not contain an $8$th root of unity. We will write $N \simeq G_2 \wr \mu_8$, as in~\Cref{prelim}.
	
	Letting $\G_m$ act on $V$ by scalar multiplication, we obtain an action of $\on{Spin}_{14}\times \G_m$ on $V$. The orbit of $v$
	under this action is open and dense in $V$. Let $H$ be the stabilizer of $v$ in $\on{Spin}_{14}\times \G_m$. Consider the
	composition \[\phi:H\hookrightarrow \on{Spin}_{14} \times \G_m\xrightarrow{\on{pr}_1} \on{Spin}_{14}. \] 
	Note that $\phi$ is injective. Indeed, its kernel, the stabilizer of $v$ in $\G_m$, is trivial. Thus $H \simeq \on{Im}(\phi)$. Moreover, clearly $S \times \{ 1 \} \subset H$, and thus 
	\[ G_2 \times G_2 \subset S \subseteq\on{Im}(\phi) \simeq H. \] 
	
	We claim that $\on{Im}(\phi) \subseteq N$. Here the inclusion should 
	be understood scheme-theoretically. To prove this claim, let $R$ be a $k$-algebra and $g\in \on{Spin}_{14}(R)$ be in the image of $\phi$. Then $gv=\lambda v$ for some $\lambda\in R^{\times}$. For any $h\in S(R)$, we have \[g^{-1}hg v=\lambda g^{-1}h v=\lambda g^{-1} v=\lambda \lambda^{-1}v=v. \] This shows that $g^{-1} h g \in S(R)$ for any $h \in S(R)$. In other words, $g \in N(R)$, as claimed.
	
	We have thus shown that $\on{Im}(\phi)$ is a subgroup of $N \simeq (G_2 \times G_2) \rtimes \mu_8$ containing 
	$G_2 \times G_2$. Thus $\on{Im}(\phi)$ is the preimage of a subgroup of $\mu_8$ under the natural projection 
	$ N \to N/(G_2 \times G_2) \simeq \mu_8$.	We conclude that
	\[ H \simeq \on{Im}(\phi)\simeq (G_2\times G_2)\rtimes \mu_m \stackrel{\rm def}{=} G_2 \wr \mu_m, \]
	where $\mu_m$ is a subgroup of $\mu_8$, i.e., $m$ is a divisor of $8$. In particular, in view of our standing assumption that $\on{char}(k) \neq 2$, this shows that $H$ is smooth. (As an aside, we remark that  
	the semidirect product $(G_2\times G_2)\rtimes \mu_m$ is direct if $m = 1$, $2$ or $4$ and not direct if $m = 8$.)	
	To finish the proof, observe that,
	
	\smallskip
    (i) by \Cref{prop.g2}, $BG_2$ is stably rational.
	
	 \smallskip
    (ii) By~\Cref{birational}(b), $B(G_2 \wr \mu_{m})$ is stably birationally equivalent to $B\mu_m$.
    On the other hand, $B\mu_m$ is stably rational by Example~\ref{ex.mu}.
	
	\smallskip
	(iii) By~\Cref{stabilizer}, applied to the representation of $\on{Spin}_{14} \times \G_m$ on $V$,
    $B(\on{Spin}_{14} \times \G_m)$ is stably birationally equivalent  
    to $BH$, where $H \simeq \on{Im}(\phi) \simeq G_2 \wr \mu_{m}$. By (ii), $BH$ is stably rational, and hence, so is
    $B(\on{Spin}_{14} \times \G_m)$.
    
    (iv) By \Cref{birational}(c), $B\on{Spin}_{14}$ is stably birationally equivalent to
    $B( \on{Spin}_{14} \times \bbG_m)$. Thus $B\on{Spin}_{14}$ is stably rational.
 \end{proof}

\begin{rmk} \label{extraspecial}
Assume that $-1$ is a square in $k$, and let $D_{2m+1}, D'_{2m+1}$ be the two non-isomorphic extraspecial $2$-groups of order $2^{2m+1}$. It is shown 
in~\cite{bogomolov2013isoclinism} that $BD_{2m+1}$ and $BD'_{2m+1}$ are stably birationally equivalent. By~\cite[Corollary 6.2]{merkurjev2019rationality}, $B\on{Spin}_n$ is stably birational to $BD_{2m+1}$. Therefore, \Cref{mainthm} has the following consequence: $BD_{2m+1}$ and $BD'_{2m+1}$ are stably 
rational for any $m\leq 6$.
\end{rmk} 

\section{Coregular representations of spinor groups}
\label{sect.coregular}

In this section we present an alternative proof of~\Cref{mainthm} over $k = \bbC$ suggested to us by
G. Schwarz. 

Let $G$ be a linear algebraic group over $k$. A linear representation $\rho \colon G \to \GL(V)$ is called coregular if $k[V]^G$ is a polynomial ring over $k$. 
Coregular representations of finite groups $G$, whose order is not divisible by $\Char(k)$, are described by the celebrated theorem of Chevalley-Shephard-Todd: 
$\rho$ is coregular if and only if $\rho(G)$ is generated by pseudo-reflections. Now suppose that $G$ is a simple linear algebraic groups over $k = \mathbb C$.
In this setting irreducible coregular representations were classified by V.~Kac, V.~Popov and E.~Vinberg in~\cite{kac-popov-vinberg} and arbitrary (not necessarily irreducible)
coregular representations by G.~Schwarz in~\cite{schwarz1978representations} and (independently) by O.~Adamovich\ and\ E.~Golovina~\cite{adamovich-golovina}.
For an overview of this area of research we refer the readers to~\cite[\S8]{popov1994}.
We will not need the full classification here; we will only use the fact 
that certain specific representations of $\on{Spin}_n$ ($n = 7, 10, 11, 14$) are coregular.

Coregular representations are related to the Noether Problem via the following simple observation.

\begin{lemma} \label{lem.coregular} Suppose a smooth algebraic $k$-group $G$ has no non-trivial characters (over $k$). 
If $G$ admits a generically free coregular representation over $k$, then $BG$ is stably rational over $k$.
\end{lemma}

\begin{proof} Let $\rho \colon G \to \GL(V)$ be a representation defined over $k$.
Since $k[V]$ is a unique factorization domain, and $G$ has no non-trivial characters,
$k(V)$ is the fraction field of $k[V]^G$; see~\cite[Theorem 3.3]{popov1994}. If $V$ is coregular,
this shows that $k(V)^G$ is rational over $k$. If $V$ is also generically free, we conclude that $BG$ is
stably rational.
\end{proof}

Now recall from the Introduction that in order to prove Theorem~\ref{mainthm}, it suffices to show that $B \on{Spin}_n$
is stably rational for $n = 7$, $10$, $11$ and $14$. Over the field $k = \mathbb C$ of complex numbers, 
Theorem~\ref{mainthm} is now an immediate consequence of the following.

\begin{prop} \label{prop.coregular}
Let $k = \mathbb C$. Then $G:=\on{Spin}_n$ admits a coregular generically free representation for $n = 7, 10, 11, 14$. 
\end{prop}

\begin{proof} Let $V_n$ be the natural representation of $\on{Spin}_n$ (via the projection $\on{Spin}_n \to \on{SO}_n$), and $W_n$ be the spin representation. If $n$ is even, let
$W_n^{1/2}$ be the half-spin representation. The following representations are shown to be coregular in \cite{schwarz1978representations}: 

\smallskip
(i) $V_7^{\oplus 3} \oplus W_7$, see Table 3a.9;

\smallskip
(ii) $V_{10}^{\oplus 5}\oplus W_{10}^{1/2}$, see Table 3a.20;

\smallskip
(iii) $V_{11}^{\oplus 4}\oplus W_{11}$, see Table 3a.25;

\smallskip
(iv) $V_{14}^{\oplus 3}\oplus W_{14}^{1/2}$, see Table 3a.31.

\smallskip
\noindent
It thus remains to show that these representations are generically free.

\smallskip
    (i)  Let $w\in W_7$ be a point in general position. By \cite[Table 1]{garibaldi2017spinors}, the stabilizer of $w$ is isomorphic to $G_2$.  We claim that 
    \begin{equation} \label{e.G2}
    \text{the $G_2$-action on $V_7^{\oplus 3}$ is generically free.}
    \end{equation}
    Note that $\Lambda = \{ x \in V_7^{\oplus 3} \oplus W_7 \, | \, \Stab_{\on{Spin}_7}(x) = 1 \}$ is a constructible subset of $V_7^{\oplus 3} \oplus W_7$. If the claim is established, then by the Fiber Dimension Theorem, \[ \dim(\Lambda) = \dim((V_7)^3 \times W_7). \] Thus
    $\Lambda$ contains a dense open subset of $V_7^3 \times W_7$, i.e., the $\on{Spin}_7$-action on  $V_7^3 \times W_7$ is generically free.
    
    It remains to prove the claim. Since $\on{dim}(V_7)=7$, the $G_2$-action on $V_7$ is the octonion representation of $G_2$. 
    In other words, we may identify $V_7$ with the space of trace-zero octonions in such a way that $G_2$ acts via octonion automorphisms. 
    If $x, y, z$ generate the octonions as a $\mathbb C$-algebra, then clearly the $G_2$-stabilizer of $(x, y, z)\in V_7^{\oplus 3}$ is trivial. The set $U \subset V_7^{\oplus 3}$ of generating triples is readily seen to be open. Note also that $U$ is non-empty, because $(i, j, k) \in U$, where $i$, $j$ and $k$ are the standard generators.
    
    \smallskip
    (ii) If $v:=(v_1,v_2,v_3,v_4,v_5)\in V_{10}^{\oplus 5}$ is a general point, the stabilizer of $v$ is isomorphic to $\on{Spin}_5$. Arguing as in (i), it remains to show that the restriction of
    $W_{10}^{1/2}$ to $\on{Spin}_5$ is generically free. 
    Recall that, for every $n\geq 1$, 
    
    \begin{itemize}
    \item
    the restriction of the spin representation $W_{2n+1}$ to $\on{Spin}_{2n}$ is $W_{2n}^{1/2} \oplus W_{2n}^{1/2}$, and 
    
    \item
  the restriction of $W_{2n}^{1/2}$ to $\on{Spin}_{2n-1}$ is the spin representation $W_{2n-1}$; 
    \end{itemize}
    see \cite[p. 1000]{igusa1970classification} or \cite[Ex. 31.2, Ex. 31.3]{bump2013lie}. Restricting $W_{10}^{1/2}$ from $\on{Spin}_{10}$ to $\on{Spin}_9$, then to $\on{Spin}_8$, etc., we see that
    as a $\on{Spin}_5$-representation, $W_{10}^{1/2}$ is isomorphic to $W_5^{\oplus 4}$. Note that $\dim(W_5) = 4$.
    There is an accidental isomorphism $\on{Spin}_5\cong \on{Sp}_4$, under which $W_5$ corresponds to the natural $4$-dimensional representation of $\on{Sp}_4$; see \cite[Proposition 5.1]{adams1996lectures}. Thus we can identify the $\on{Spin}_5$-action on $W_5^4$ with the $\on{Sp}_4$-action on $\Mat_{4 \times 4}$ via left multiplication.
    The latter action is clearly generically free.
    
    \smallskip
    (iii) If $v:=(v_1,v_2,v_3,v_4)\in V_{11}^{\oplus 4}$ is a general point, the stabilizer of $v$ is isomorphic to $\on{Spin}_7$. Once again, it suffices to show that when we restrict $W_{11}$ from $\on{Spin}_{11}$ to $\on{Spin}_7$, it becomes generically free.
    Arguing as in (ii) above, we see that as a $\on{Spin}_7$-representation, $W_{11}$ is isomorphic to $W_7^{\oplus 4}$. If $w \in W_7$ is a general point, then
    the $\on{Spin}_7$-stabilizer of $w$ is isomorphic to $G_2$; see~\eqref{e.7}.
    Thus it suffices to show that the $G_2$-action on $W_7^{\oplus 3}$ is generically free.
    As a $G_2$-reresentation, $W_7 = \mathbb C \oplus V_7$, where $\mathbb C$ is the trivial $1$-dimensional representation generated by the vector stabilized by $G_2$, and
    $V_7$ is a $7$-dimensional representation, which can be identified with trace-zero octonions, as in (i). The desired conclusion now follows from~\eqref{e.G2}.
    
    \smallskip  
    (iv) If $w\in W_{14}^{1/2}$ is a general point, then by \cite[Table 1]{garibaldi2017spinors} the stabilizer of $w$ in $\on{Spin}_{14}$ is isomorphic to $G_2\times G_2$. 
    It thus suffices to show that $G_2 \times G_2$ acts generically freely on $V_{14}^{\oplus 3}$.
    Since $G_2\times G_2$ has trivial center, we may view it as a subgroup of $\on{SO}(V_{14})$ via the projection $\on{Spin}_{14}\to \on{SO}(V_{14})$. 
    By~\cite[\S 8]{garibaldi2017spinors}, $V_{14} \cong V_7 \oplus V_7$ as a $G_2\times G_2$-representation. 
    Here $V_7$ is identified with trace-zero octonions, with the natural action of $G_2$, and $(g, g') \in G_2 \times G_2$
    acts on $(v, v') \in V_7 \times V_7$ by $(v, v') \mapsto (g(v), g'(v'))$. By~\eqref{e.G2}, $G_2$ acts generically freely on $V_7^3$. Hence,
    $G_2\times G_2$ acts generically freely on $V_{14}^{\oplus 3} \simeq V_7^{\oplus 3} \oplus V_7^{\oplus 3}$, as desired.
\end{proof}

	\section*{Acknowledgments} 
The authors are grateful to Gerald Schwarz for contributing Proposition~\ref{prop.coregular}, to
Skip Garibaldi and Mattia Talpo for informative correspondence, and the anonymous referee for helpful comments.

\end{document}